\pdfoutput=1 
\documentclass{article}
\usepackage{PRIMEarxiv}
\usepackage{comment}
\usepackage[utf8]{inputenc} 
\usepackage[T1]{fontenc}    
\usepackage{url}            
\usepackage{booktabs}       
\usepackage{amsfonts}       
\usepackage{nicefrac}       
\usepackage{microtype}      
\usepackage{fancyhdr}       
\usepackage{graphicx}       
\graphicspath{{media/}}     

\usepackage[textwidth=3cm]{todonotes}
\usepackage{xcolor}

\usepackage{tikz-cd}
\usepackage{amssymb, amsmath}
\usepackage{amsthm}
\usepackage{enumitem}
\usepackage{hyperref}
\hypersetup{
  colorlinks=true,
}
\hypersetup{
    linkcolor=magenta,
    filecolor=magenta,      
    urlcolor=cyan,
    citecolor=magenta,
    }

\theoremstyle{remark}
\newtheorem{remark}{Remark}[section]

\theoremstyle{definition}
\newtheorem{theorem}[remark]{Theorem}
\newtheorem{lemma}[remark]{Lemma}
\newtheorem{corollary}[remark]{Corollary}
\newtheorem{proposition}[remark]{Proposition}

\newtheorem{definition}[remark]{Definition}
\newtheorem{example}[remark]{Example}

\title{Surfaces associated with first-order ODEs
\thanks{\textit{\underline{Citation}}: 
\textbf{Authors. Title. Pages.... DOI:000000/11111.}} 
}

\author{
  *A. J. Pan-Collantes, J. A. Álvarez-García\\
  Departamento de Matem\'aticas\\ 
  Universidad de C\'{a}diz - UCA \\
  Puerto Real\\
  \texttt{\{antonio.pan@uca.es, jose.alvg@gmail.com\}} \\
}

\begin{document}
\newpage
\maketitle

\begin{abstract}
A link between first-order ordinary differential equations (ODEs) and 2-dimensional Riemannian manifolds is explored. Given a first-order ODE, an associated Riemannian metric on the variable space is defined, and some properties of the resulting surface are studied, in relation to the integrability of the equation. 

Next, deformations of the associated surfaces are considered. A relation between relative Jacobi fields on the deformed surface and integrability of the ODE is established, showing that this class of vector fields are useful for solving first-order ODEs. As a consequence, it is proven that if the associated surface is of constant curvature, or alternatively it can be deformed into one of constant curvature, then the ODE can be integrated by quadratures. In particular, the search for an integrating factor for the ODE is interpreted as a deformation of the associated surface into a flat one.
\end{abstract}



\section{Introduction}

First-order ordinary differential equations (ODEs) are foundational in many scientific fields, enabling the modeling and understanding of numerous phenomena including concentration and dilution problems, population models and even astrophysical research. Moreover, as the simplest form of ODEs, they serve as a natural starting point for more advanced research.

The works of S. Lie and E. Cartan \cite{lie1888classification,cartan1899certaines,cartan1924varietes} provided a geometric perspective on differential equations. Their approach, which turned out to be particularly useful in understanding and solving nonlinear ODEs, gave rise to the concepts of symmetry and invariance. This interplay between differential equations and geometric objects has been extensively studied since then, and extends to recent decades \cite{stiller1981differential, crampin1996linear,frittelli2001differential,newman2003projective,nurowski2005differential,tsamparlis2010lie,doubrov2016geometry,Bayrakdar2018a,Bayrakdar2018,autonomous2ODE,integrationJacobifields}.    

In this work, we explore an approach aligned with the perspective provided in \cite{Bayrakdar2018a, Bayrakdar2018,autonomous2ODE,integrationJacobifields}. We introduce a geometric framework in which an ODE is associated with a surface, i.e., a 2-dimensional Riemannian manifold. This correspondence allows us to explore the integrability of the ODE through geometric properties of the associated surface, such as geodesics and curvature (Section \ref{secBay}). An important role of the curvature should be expected, due to its pivotal function in both mathematics and physics. And this is indeed the case: we show that within our framework, the curvature is intimately related to the integrability of the ODE. In particular, we establish that if the associated surface is flat, then the ODE can be integrated by quadratures. 

Furthermore, we consider deformations of the associated surface, and study the implications of these deformations on the integrability of the original ODE (Section \ref{secDeform}). We show that the knowledge of a relative Jacobi field on any of the deformed surfaces (including the original surface) enables the integration of the equation. As a consequence, we prove that deforming the surface into one with constant curvature leads to the integration of the ODE. In this framework, the classical search for an integrating factor is interpreted as a deformation of the associated surface into a flat one.

\section{Preliminaries}\label{preliminaries}
We begin with a brief review of the main concepts and results relevant to this paper, along with the necessary notation. Given a first-order ODE in the form  
\begin{equation}\label{ODE}  
    \frac{du}{dx}=\phi(x,u),  
\end{equation}  
where \( \phi \) is a smooth function defined on an open subset \( U\subseteq \mathbb{R}^2 \), a classical approach to study this equation is through the associated vector field $A$, defined for $(x,u)\in U$ as
\begin{equation}\label{campoA}  
    A_{(x,u)}=\partial_x+\phi(x,u)\partial_u.
\end{equation}  

To analyze solutions geometrically, we introduce the following notation: given a real-valued smooth function $f: I\to \mathbb R$, with $I$ an open interval containing 0, we denote by
\begin{equation}\label{curve}  
    \gamma_f(t)=(t,f(t)),
\end{equation}
the standard parametrization of the graph of $f$.

An explicit solution $u=f(x)$ of equation \eqref{ODE} corresponds to a curve \eqref{curve} that is an integral curve of $A$, i.e., 
$$
\dot{\gamma}_f(t)=A_{\gamma_f(t)}.
$$
Conversely, any integral curve of $A$ can be reparametrized to the form $\gamma_f$, for a certain solution $f$ of \eqref{ODE} \cite{arnold1992ordinary,stephani}.

The distribution generated by $A$ is spanned by the vector fields annihilated by the 1-form
\begin{equation}\label{1form}  
    \omega=-\phi(x,u)dx+du.
\end{equation}

A first integral of $\omega$, i.e., a smooth function $F\in \mathcal{C}^{\infty}(U)$ satisfying $dF=\mu \omega$ locally, for some non-vanishing smooth function $\mu$, provides an implicit description of the solutions of \eqref{ODE} \cite{arnold1992ordinary,stephani}:
$$
F(x,u)=C, \,C \in \mathbb R.
$$
The function $\mu$ is called an integrating factor.

Now we recall some fundamental facts on Riemannian geometry. Throughout the text, the term \emph{surface} will be used to indicate a pair $(\mathcal{S},g)$ consisting of a 2-dimensional manifold $\mathcal{S}$ equipped with a Riemannian metric $g$, that is, a two-times covariant symmetric tensor field, positive definite, and hence non-degenerate. We will refer to the surface simply by $\mathcal{S}$, whenever the context is clear.

Given an orthonormal frame $\{e_1,e_2\}$ in $\mathcal{S}$ with dual coframe $\{\omega_1,\omega_2\}$, the corresponding Riemannian metric can be written as \cite{lee2006curvature}
\begin{equation}\label{metrictensor}
    g=\omega_1 \otimes \omega_1+\omega_2 \otimes \omega_2.
\end{equation}

It is a well-known fact that every surface $\mathcal{S}$ is endowed with a
uniquely determined torsionless metric connection $\nabla$, called the Levi-Civita connection. In the provided frame, this connection is described by the matrix of 1-forms
\begin{equation}\label{connectionform}
\Theta=\begin{pmatrix}
0&-T_{12}^1 \omega^1-T_{12}^2\omega^2\\ 
T_{12}^1 \omega^1+T_{12}^2 \omega^2&0\\
\end{pmatrix},
\end{equation}
where $T_{ij}^k$ are the structure coefficients of the coframe, i.e., smooth functions such that $d\omega^k=T_{ij}^k \omega^i \wedge \omega^j$. The details can be found, for instance, in \cite{chen1999lectures,Morita,ivey2016cartan}.

Consider now two vector fields described by their components in the frame $\{e_1,e_2\}$:
$$
X=
\begin{pmatrix}
    x_1\\
    x_2
\end{pmatrix}, \qquad
Y=
\begin{pmatrix}
    y_1\\
    y_2
\end{pmatrix},
$$
where $x_i,y_i$ are smooth functions defined on $\mathcal{S}$. The covariant derivative $\nabla_{X} Y$ can be expressed in this frame as
$$
\nabla_{X} Y=
\begin{pmatrix}
X(y_1)\\
X(y_2)
\end{pmatrix} +i_X \Theta \cdot \begin{pmatrix}
y_1\\
y_2\\
\end{pmatrix}.
$$
Here, \(i_X \Theta\) is a matrix whose entries represent the interior product of \(X\) with each entry of \(\Theta\), and the dot denotes matrix multiplication. Therefore
\begin{align}
    \nabla_{X} Y &= \begin{pmatrix}
    X(y_1) \\
    X(y_2)
    \end{pmatrix} + 
    \begin{pmatrix}
    0 & -T_{12}^1 x_1 - T_{12}^2 x_2 \\ 
    T_{12}^1 x_1 + T_{12}^2 x_2 & 0
    \end{pmatrix}
    \cdot
    \begin{pmatrix}
    y_1 \\
    y_2
    \end{pmatrix} \nonumber \\
    &= \begin{pmatrix}
    X(y_1) - T_{12}^1 x_1 y_2 - T_{12}^2 x_2 y_2 \\
    X(y_2) + T_{12}^1 x_1 y_1 + T_{12}^2 x_2 y_1
    \end{pmatrix}. \label{covDeriv}
\end{align}

Also, given a curve $\alpha: I\to \mathcal{S}$, where $I$ is an open interval in $\mathbb R$, and a vector field $Y$ defined along $\alpha$, the covariant derivative $\nabla_{\dot{\alpha}(t)} Y(t)$ for $t\in I$ is given by the expression
\begin{equation}\label{eq:covDerivCurve}
    \nabla_{\dot{\alpha}(t)} Y(t)= \begin{pmatrix}
        y'_1(t) - T_{12}^1 x_1(t) y_2(t) - T_{12}^2 x_2(t) y_2(t) \\
        y'_2(t) + T_{12}^1 x_1(t) y_1(t) + T_{12}^2 x_2(t) y_1(t)
        \end{pmatrix},    
\end{equation}
where $x_1(t),x_2(t)$ are the components of the tangent vector $\dot{\alpha}(t)$ in the frame $\{e_1,e_2\}$, and $y_1(t),y_2(t)$ are the components of $Y(t)$ in the same frame \cite{lee2006curvature}.

A fundamental feature of a surface is its Gaussian curvature $\mathcal{K}$. The connection form \eqref{connectionform} allows us to obtain it by means of Gauss equation \cite{chen1999lectures,ivey2016cartan}:
\begin{equation}\label{eq:curvature}
    \mathcal{K}=d\Theta^1_2(e_1,e_2).
\end{equation}

In curved spaces, the notion of a geodesic is used to extend the idea of a straight line in flat spaces. We recall the following definitions:
\begin{definition}
A curve $\gamma: I\subseteq \mathbb R \to \mathcal{S}$ is called a geodesic if
$$
\nabla_{\dot{\gamma}(t)}\dot{\gamma}(t)=0
$$
for every $t\in I$.
\end{definition}

\begin{definition}
A vector field $X$ will be called a geodesic vector field if 
$$
\nabla_{X}X=0.
$$

\end{definition}

It turns out that a vector field is a geodesic vector field if and only if its integral curves are geodesics of the surface \cite{lee2006curvature,docarmoriemannian,Carinena2023}.

An intimately related concept is that of Jacobi field, which plays a crucial role in understanding how geodesics evolve, capturing the curvature of the surface and providing a local view of geodesic variations. Essentially, they offer insights into the behavior of nearby geodesics \cite{lee2006curvature}.

\begin{definition}
A vector field $J$ along a geodesic $\gamma$ is called a Jacobi field if it satisfies the Jacobi equation:
\begin{equation}\label{eq:Jacobi}
\nabla_{\dot{\gamma}} \nabla_{\dot{\gamma}} J+\mathcal{K}\left( g(\dot{\gamma},\dot{\gamma})J-g(J,\dot{\gamma})\dot{\gamma}\right)=0,
\end{equation}
for every point $p$ in the image of $\gamma$. 
\end{definition}


The following definition is a natural extension of the concept of Jacobi field to vector fields defined on the surface, not necessarily along a geodesic:
\begin{definition}\label{def:Jacobi}
Let $X$ be a geodesic vector field defined on a surface $\mathcal S$. A vector field $J$ on $\mathcal{S}$ such that the restriction of $J$ to each integral curve $\gamma$ of $X$ is a Jacobi vector field along $\gamma$ will be called a Jacobi vector field relative to $X$.
\end{definition}
\begin{remark}\label{rem:jacobi}
    As it can be deduced from equation \eqref{eq:Jacobi}, $J$ is a Jacobi vector field relative to $X$ if and only if
    \begin{equation}\label{eq:Jacobirel}
        \nabla_{X} \nabla_{X} J+\mathcal{K}\left( g(X,X)J-g(J,X)X\right)=0
        \end{equation}
for every point $p\in \mathcal S$. 
\end{remark}

Finally, for the sake of completeness, we state here a well-known result concerning the commutativity of vector fields \cite[Theorem 13.10]{lee2013smooth}, which will be used later:
\begin{theorem}\label{commutingvf}
Two pointwise linearly independent vector fields $Y_1,Y_2$ defined on $U$ satisfy $[Y_1,Y_2]=0$ if and only if, for every $p\in U$, there exists a local coordinate change
$$
\begin{array}{rccc}
\varphi:& W &\to & V\\
        & (s,t) & \mapsto &(x,u)
\end{array},
$$
with $(0,0) \in W\subseteq \mathbb R^2$ and $p\in V \subseteq U$, such that $\varphi(0,0)=p$ and 
$$
\varphi_*\left(\frac{\partial}{\partial s}\right) =Y_1,\quad \varphi_*\left(\dfrac{\partial}{\partial t}\right)=Y_2.
$$  
\end{theorem}

\section{Surface associated to a first-order ODE}\label{secBay}

In this section, we revisit an approach for assigning a surface $\mathcal{S}$ to a given first-order ODE, as introduced in \cite{Bayrakdar2018a,Bayrakdar2018,autonomous2ODE,integrationJacobifields}, and analyze some of its key properties. To avoid clutter, we omit the arguments of the function $\phi$ and its partial derivatives $\phi_x, \phi_u$ when clear from the context.

\begin{definition}\label{def:asurf}
Given a first-order ODE like \eqref{ODE}, we define the associated surface $\mathcal{S}$ to be the open set $U$ together with the Riemannian metric given by
\begin{align}\label{metricBay}
    g&=( 1+ \phi^2)dx\otimes dx-\phi dx\otimes du-\phi du\otimes dx+du\otimes du,
\end{align}
or in matrix form
\begin{equation}
    g=\begin{pmatrix}
    1+ \phi^2 & -\phi \\
    - \phi  & 1
    \end{pmatrix}.
\end{equation}
\end{definition}

\begin{remark}\label{rem:prop}
The reader can easily check that, with this metric, the pair of vector fields $\{A, \partial_u\}$ constitutes an orthonormal frame for the surface $\mathcal{S}$. The corresponding dual coframe is given by the 1-forms
\begin{align*}
& \omega^1=dx,\\
& \omega^2=-\phi dx+du.
\end{align*}
\end{remark}

    

To explore the geometry of the surface $\mathcal{S}$, we will first calculate the connection 1-form of the corresponding Levi-Civita connection. Observe that $d \omega^1=0$ and ${d\omega^2=\phi_u dx\wedge du=\phi_u \omega^1\wedge \omega^2}$, so according to equation \eqref{connectionform} we have:
$$
\Theta=
\begin{pmatrix}
0&-\phi_u \omega^2\\ 
\phi_u \omega^2&0\\
\end{pmatrix}
=
\begin{pmatrix}
0&\phi \phi_u dx -\phi_u du\\ 
-\phi \phi_u dx + \phi_u du&0\\
\end{pmatrix}.
$$

We are now in position to state and prove the following result, regarding the geodesics of the surface $\mathcal{S}$:
\begin{proposition}\label{propGeod}
Consider an ODE like \eqref{ODE} such that $\phi_u\neq 0$, and let $f$ be a smooth function. Then the curve $\gamma_f$ is a geodesic of $\mathcal{S}$ if and only if $f$ is a solution of \eqref{ODE}.
\end{proposition}
\begin{proof}










First, observe that the components of the tangent vector $\dot{\gamma}_f(t)$ in the frame $\{A,\partial_u\}$ are
\begin{equation}\label{tangent}
\dot{\gamma}_f(t)=\begin{pmatrix}
    1\\
    f'(t)-\phi\\
\end{pmatrix},
\end{equation}
and then, using equation \eqref{eq:covDerivCurve}, we obtain 
\begin{equation}\label{covDerivPart}
    \nabla_{\dot{\gamma}_f(t)}\dot{\gamma}_f(t)=\begin{pmatrix}
        - \phi_u(f'(t)-\phi)^2\\
        -\phi_u \phi-\phi_x+f''(t)\\
    \end{pmatrix}.    
\end{equation}

If $f$ is a solution of \eqref{ODE} then $f'(t)=\phi(t,f(t))$. Differentiating both sides with respect to $t$ gives
$$
f''(t)=\phi_x+\phi_uf'(t)=\phi_x+\phi_u\phi.
$$

Hence, by substituting in \eqref{covDerivPart} we obtain
$$
\nabla_{\dot{\gamma}_f(t)}\dot{\gamma}_f(t)=
\begin{pmatrix}
    0\\
    0\\
\end{pmatrix},
$$
so $\gamma_f$ is a geodesic.

Conversely, provided that $\gamma_f$ is a geodesic of $\mathcal{S}$ we have that $\nabla_{\dot{\gamma}_f(t)}\dot{\gamma}_f(t)=0$, and then by \eqref{covDerivPart}
\begin{subequations}\label{systGeod}
\begin{align}
\phi_u\left( f'(t)- \phi\right)^2&=0,\label{systGeod1}\\
f''(t) - \phi \phi_u - \phi_x &= 0.\label{systGeod2}
\end{align}
\end{subequations}

Under the assumption $\phi_u\neq 0$, equation \eqref{systGeod1} reduces to 
\begin{equation}\label{eq1}
f'(t)-\phi=0,    
\end{equation}
and therefore $f$ is a solution of \eqref{ODE}.
\end{proof}
\begin{remark}
Observe that, in particular, this proposition shows that the vector field $A$ is a geodesic vector field (this result appears also in  \cite{Bayrakdar2018a}). Moreover, it is the only geodesic vector field $V$ such that $i_{\partial_x} V=1$.
\end{remark}

An important geometric quantity associated to a surface is its Gaussian curvature. We will calculate it for the surface $\mathcal{S}$.
\begin{proposition}\label{curvat0}
The Gaussian curvature of the surface $\mathcal{S}$ is given by the expression
\begin{equation}\label{eq:curvat0}
     \mathcal{K}=-\partial_u(A(\phi)).
\end{equation}
\end{proposition}
\begin{proof}
Considering equation \eqref{eq:curvature} we have that
\begin{align*}
\mathcal{K} &= d\Theta^1_2(A,\partial_u) \\
  &= d(\phi \phi_u dx -\phi_u du)(A,\partial_u) \\
  &= \left(-\phi_u^2-\phi\phi_{uu}-\phi_{xu}\right)dx\wedge du(A,\partial_u) \\
  &= (-\phi \phi_u-\phi_x)_u \\
  &= -\partial_u(A(\phi)).
\end{align*}
\end{proof}




The expression obtained in Proposition \ref{curvat0} allows us to establish a link between the flatness of the surface and the integrability of the corresponding ODE:
\begin{theorem}\label{th:integr}
Given the first-order ODE \eqref{ODE}, if the Gaussian curvature $\mathcal{K}$ of the associated surface $\mathcal{S}$ is zero, then the ODE can be fully integrated.
\end{theorem}

\begin{proof}

Since $\mathcal{S}$ is flat, it follows that $\mathcal{K}=-\partial_u(A(\phi))=0$, so $A(\phi)$ is a smooth function which does not depend on the variable $u$. We now consider $\Psi$ to be any smooth function on $x$ satisfying
$$
\dfrac{d\Psi}{dx}=A(\phi).
$$

Then, the expression $F(x,u)=\phi(x,u)-\Psi(x)$ is a first integral of \eqref{campoA}. In fact, we have that 
$$
A(F)=A(\phi-\Psi)=A(\phi)-\dfrac{d\Psi}{dx}=0.
$$

So the general solution of \eqref{ODE} is given by 
$$
\phi(x,u)-\Psi(x)=C,
$$
with $C\in \mathbb R$.
\end{proof}

\begin{example}
Consider the first-order ODE
\begin{equation}\label{odeEx}
    (1-x)\frac{du}{dx}=e^{(x-1)\frac{du}{dx}-u}+1.
\end{equation}
We can put this equation in normal form
$$
\dfrac{d u}{dx}=\dfrac{\mathbf W(e^{-u-1})+1}{1-x},
$$
defined on $U=\{(x,u)\in \mathbb R^2:x\neq 1\}$. Here, $\mathbf W$ is the smooth function that satisfies
$$
\mathbf W(x)e^{\mathbf W(x)}=x,
$$
known as the Lambert function \cite{handbookfunctions}. Recall that 
$$
\dfrac{d\mathbf W(x)}{dx}=\frac{1}{e^{\mathbf W(x)}\left(1+\mathbf W(x)\right)}.
$$
The associated vector field is 
$$
A=\partial_x+\dfrac{\mathbf W(e^{-u-1})+1}{1-x}\partial_u,
$$
and the curvature $\mathcal{K}$ of the associated surface is:
\begin{align*}
\mathcal{K} & = -\partial_u\left(A\left(\frac{\mathbf{W}\left(e^{-u-1}\right) + 1}{1 - x}\right)\right) \\
  & = -\partial_u\left(\frac{1}{(x-1)^2}\right) \\
  & = 0.
\end{align*}

Since this surface is flat, equation \eqref{odeEx} can be integrated following the proof of Theorem \ref{th:integr}. Consider the function $\Psi(x)=\frac{1}{1-x}$, which satisfies 
$$
\frac{d\Psi}{dx}=\frac{1}{(x-1)^2}=A\left(\phi\right).
$$ 
Then, a first integral of $A$ is
\begin{align*}
    F(x,u)&=\phi(x, u) - \Psi(x) \\
    &= \dfrac{\mathbf W(e^{-u-1})+1}{1-x}-\dfrac{1}{1-x}\\
    &=\dfrac{\mathbf{W}(e^{-u-1})}{1-x}.
\end{align*}

Finally, we can isolate $u$ in 
$$
\dfrac{\mathbf{W}(e^{-u-1})}{1-x}=C,
$$
where $C$ is a nonzero constant, obtaining the general solution 
$$
u(x)=-\ln(C(1-x))-C(1-x)-1,
$$
defined for $C\neq 0$ and $x\in \mathbb R$ such that $C(1-x)>0$.
\end{example}


\vspace{1cm}

While we have established that zero curvature of the associated surface leads to integrability of the original ODE, this condition is rarely satisfied in practice. In order to extend this result to a broader class of equations, we will introduce in the next section a conformal deformation of the metric associated to the ODE, which leads to our main result: if there exists a deformation into a constant curvature surface, then the ODE is integrable.

\section{Deformation of the associated surface}\label{secDeform}

In this section, we explore the idea of deforming the surface defined in the previous section while preserving its essential features in relation to equation \eqref{ODE}, and we study the role of this deformation in the integrability of the equation. As we will see, the deformation into a zero curvature surface is in correspondence with the finding of an integrating factor for the ODE. Furthermore, we will show that this is not the only suitable deformation that enables integrability.

In order to define a deformation of the surface $\mathcal{S}$, we relax the condition in Remark \ref{rem:prop} of $\{A,\partial_u\}$ being an orthonormal frame, and consider the family $\mathcal{G}$ of all Riemannian metrics in which $A$ and $\partial_u$ are orthogonal and $A$ has unit length. This family of metrics can be \emph{indexed} by the set $\mathcal{C}^{\infty}(U)$ of smooth functions defined on $U$, as stated in the following proposition:
\begin{proposition}
Every Riemannian metric in the family $\mathcal{G}$ can be expressed in coordinates $(x,u)$ as the matrix
\begin{equation}\label{deformations}
g_{\epsilon}=\begin{pmatrix}
1+ \phi^2 e^{2\epsilon} & -\phi e^{2\epsilon}\\
- \phi e^{2\epsilon}  & e^{2\epsilon} 
\end{pmatrix},
\end{equation}
with $\epsilon=\epsilon(x,u) \in\mathcal{C}^{\infty}(U)$.
\end{proposition}

\begin{proof}
Consider any Riemannian metric $\tilde{g}$ in $\mathcal{G}$. The vector field $ \tilde{g}(\partial_u,\partial_u)^{-\frac{1}{2}} \partial_u$ has unit length. Hence, the pair of vector fields 
$$
\left\{A, \tilde{g}(\partial_u,\partial_u)^{-\frac{1}{2}} \partial_u\right\}
$$
forms an orthonormal frame with respect to $\tilde{g}$.

By defining $\epsilon(x,u):=\ln \left(\sqrt{\tilde{g}(\partial_u,\partial_u)}\right)$, this frame can be rewritten as
\begin{equation}\label{eq:ortframe}
\{A,e^{-\epsilon}\partial_u\},
\end{equation}
and the corresponding orthonormal dual coframe, denoted by $\{\omega_{\epsilon}^1,\omega_{\epsilon}^2\}$, is given by
\begin{equation}\label{eq:cofdef}
\begin{aligned}
& \omega_{\epsilon}^1=dx,
\\
& \omega_{\epsilon}^2=-e^\epsilon \phi dx+e^\epsilon du.
\end{aligned}    
\end{equation}

Then, the Riemannian metric $\tilde{g}$ can be expressed, according to equation \eqref{metrictensor}, as:
\[
\begin{aligned}
\tilde{g} &= \omega_{\epsilon}^1 \otimes \omega_{\epsilon}^1 + \omega_{\epsilon}^2 \otimes \omega_{\epsilon}^2\\
&= dx \otimes dx + \left(-e^\epsilon \phi dx+e^\epsilon du\right) \otimes \left(-e^\epsilon \phi dx+e^\epsilon du\right)\\
&= dx \otimes dx + \phi^2 e^{2\epsilon} dx \otimes dx - \phi e^{2\epsilon} dx \otimes du - \phi e^{2\epsilon} du \otimes dx + e^{2\epsilon} du \otimes du\\
&= (1 + \phi^2 e^{2\epsilon}) dx \otimes dx - \phi e^{2\epsilon} dx \otimes du - \phi e^{2\epsilon} du \otimes dx + e^{2\epsilon} du \otimes du.
\end{aligned}
\]

In matrix form, this becomes:
\begin{equation}\label{metrigepsilon}
    \tilde{g}= \begin{pmatrix}
        1+ \phi^2 e^{2\epsilon} & -\phi e^{2\epsilon}\\
        - \phi e^{2\epsilon}  & e^{2\epsilon}
        \end{pmatrix}.    
\end{equation}

\end{proof}

\begin{remark}
    Let us denote by $g_{\epsilon}$ the metric given by the matrix \eqref{metrigepsilon} and by $\mathcal{S}_{\epsilon}$ the surface $(U, g_{\epsilon})$. Observe that the surface $\mathcal{S}$ from Definition \ref{def:asurf} is clearly the member of the family $\mathcal{G}$ corresponding to the choice $\epsilon = 0$, that is, $\mathcal{S} = \mathcal{S}_0$. Hence, the set  
    $$
    \{\mathcal{S}_{\epsilon} : \epsilon \in \mathcal{C}^{\infty}(U)\}
    $$  
    can be regarded as a \emph{deformation} of $\mathcal{S}$.  
\end{remark}

Now, we proceed to explore the geometry of the deformed surfaces $\mathcal{S}_{\epsilon}$. Recall that Proposition \ref{propGeod} establishes that the solutions to equation \eqref{ODE} correspond to geodesics of $\mathcal{S}_0$. Notably, this set of geodesics remains as geodesics through the deformation. In what follows, we will use the notation
\begin{equation}\label{deltaeps}
    \Delta_{\epsilon}:=A(\epsilon)+\phi_u
\end{equation}
to simplify the formulas.

\begin{proposition}\label{geodesicDef}
The vector field $A$ is a geodesic vector field for every surface $\mathcal{S}_{\epsilon}$.
\end{proposition}
\begin{proof}
Taking into account that 
\begin{align*}
& d\omega_{\epsilon}^1=0,
\\
& d\omega_{\epsilon}^2=(\epsilon_x+\epsilon_u \phi+ \phi_u) \omega_{\epsilon}^1 \wedge \omega_{\epsilon}^2=\Delta_{\epsilon} \omega_{\epsilon}^1 \wedge \omega_{\epsilon}^2,
\end{align*}
and according to \eqref{connectionform}, we can write the connection form for the Levi-Civita connection of $S_{\epsilon}$ with respect to the orthonormal frame \eqref{eq:ortframe} as
\begin{equation}\label{LCdef}
\Theta=\begin{pmatrix}
    0 & -\Delta_{\epsilon} \omega_{\epsilon}^2\\ 
    \Delta_{\epsilon} \omega_{\epsilon}^2 & 0
\end{pmatrix}.    
\end{equation}

Therefore, $\nabla_{A} A=0$, and $A$ is a geodesic vector field.
\end{proof}

To relate the Gaussian curvature of the deformed surfaces, $\mathcal{K}_{\epsilon}$, and the integrability of \eqref{ODE}, we first need an explicit expression for $\mathcal{K}_{\epsilon}$:
\begin{proposition}\label{prop:curvature}
The Gaussian curvature of $\mathcal{S}_{\epsilon}$ is given by the expression
\begin{equation}
    \mathcal{K}_{\epsilon}=-A(\Delta_{\epsilon})-\Delta_{\epsilon}^2.
\end{equation}   

\end{proposition}

\begin{proof}
Note that
$$
\begin{aligned}
    d\Theta_2^1 &= -d\Delta_{\epsilon} \wedge w^2 - \Delta_{\epsilon} \wedge dw^2 \\
    &= -\left( (\Delta_{\epsilon})_x + (\Delta_{\epsilon})_u \phi + \Delta_{\epsilon}^2 \right) w^1 \wedge w^2 \\
    &= -\left( A(\Delta_{\epsilon}) + \Delta_{\epsilon}^2 \right) w^1 \wedge w^2,
\end{aligned}
$$
and according to equation \eqref{eq:curvature}
$$
\mathcal{K}_{\epsilon}=d\Theta_2^1(A,e^{-\epsilon}\partial_u)=-A(\Delta_{\epsilon})-\Delta_{\epsilon}^2.
$$
\end{proof}

Before proceeding to the main result of the paper, we will introduce two operators that will play a crucial role in the proofs of the results. 
\begin{definition}\label{def:oper}
Given a first-order ODE \eqref{ODE} we define two operators, $\mathfrak{T}_{\epsilon}$ and $\mathfrak{S}_{\epsilon}$, as follows:
\begin{equation}
    \begin{aligned}
        \mathfrak{T}_{\epsilon}(h)&:=A(h)+\Delta_{\epsilon}h,\\
        \mathfrak{S}_{\epsilon}(h)&:=A(h)-\Delta_{\epsilon}h,\\
    \end{aligned}
\end{equation}
where $h\in \mathcal{C}^{\infty}(U)$.
\end{definition}

These operators satisfy the following properties:
\begin{lemma}\label{prop:intsymop}
Let $h\in \mathcal{C}^{\infty}(U)$. Then:
\begin{enumerate}[label=(\alph*)]
    \item \label{it:integrating} $\mathfrak{T}_{\epsilon}(h)=0$ if and only if $e^{\epsilon} h$ is an integrating factor of \eqref{1form}.
    \item If $h$ is non-vanishing in $U$, then $\mathfrak{S}_{\epsilon}(h)=0$ if and only if $e^{\epsilon} h^{-1}$ is an integrating factor of \eqref{1form}.
\end{enumerate} 
\end{lemma}
\begin{proof}
For the proof of part (a), observe that
\begin{equation}
    \begin{aligned}
        d\left(e^{\epsilon} h (-\phi dx+du)\right)
        &= (- e^{\epsilon} h \phi)_u du\wedge dx+(e^{\epsilon} h)_x dx\wedge du \\
        &= \left((e^{\epsilon} h)_u \phi+e^{\epsilon} h\phi_u+(e^{\epsilon} h)_x\right) dx\wedge du \\
        &= \left(A(e^{\epsilon} h) + \phi_u e^{\epsilon} h\right) dx \wedge du \\
        &= \left(e^{\epsilon} A(\epsilon)h + e^{\epsilon} A(h) + \phi_u e^{\epsilon}h\right)dx \wedge du \\
        &= e^{\epsilon} \left(A(\epsilon)h + A(h) + \phi_u h\right)dx \wedge du \\
        &= e^{\epsilon} \mathfrak{T}_{\epsilon}(h)dx \wedge du.
    \end{aligned}
\end{equation}
Then, $e^{\epsilon} h$ is an integrating factor of \eqref{1form} if and only if $\mathfrak{T}_{\epsilon}(h)=0$.

To prove part (b), we proceed similarly:
\begin{equation}
    \begin{aligned}
        d\left(e^{\epsilon}h^{-1} (-\phi dx+du)\right)
        &= \left(- \frac{e^{\epsilon}}{h} \phi\right)_u du\wedge dx+\left(\frac{e^{\epsilon}}{h}\right)_x dx\wedge du \\
        &= \left(\left(\frac{e^{\epsilon}}{h}\right)_u \phi+\frac{e^{\epsilon}}{h}\phi_u+\left(\frac{e^{\epsilon}}{h}\right)_x\right) dx\wedge du \\
        &= \left(A\left(\frac{e^{\epsilon}}{h}\right) + \phi_u \frac{e^{\epsilon}}{h}\right) dx \wedge du \\
        &= \frac{e^{\epsilon}A(\epsilon)h-e^{\epsilon}A(h)+e^{\epsilon}h\phi_u}{ h^2}dx \wedge du \\
        &= \frac{e^{\epsilon}}{h^2}\left(A(\epsilon)h-A(h)+h\phi_u\right)dx \wedge du \\
        &= -\frac{e^{\epsilon}}{h^2} \mathfrak{S}_{\epsilon}(h)dx \wedge du.
    \end{aligned}
\end{equation}

Therefore, $e^{\epsilon} h^{-1}$ is an integrating factor of \eqref{1form} if and only if $\mathfrak{S}_{\epsilon}(h)=0$.
\end{proof}

Additionally, these operators are related to the curvature of the surface $\mathcal{S}_{\epsilon}$, as shown in the following lemma:
\begin{lemma}\label{lem:op}
    It is satisfied that
    \begin{equation}
        \mathfrak{T}_{\epsilon}\circ \mathfrak{S}_{\epsilon}=A^2+\mathcal{K}_{\epsilon}.
    \end{equation}
\end{lemma}
\begin{proof}
    Consider a smooth function $h$ defined on $U$, then
    \begin{align*}
        \mathfrak{T}_{\epsilon}\circ \mathfrak{S}_{\epsilon} (h)&= (A+\Delta_{\epsilon})\circ (A-\Delta_{\epsilon})(h) \\
        &= A^2(h)-A(\Delta_{\epsilon} h)+\Delta_{\epsilon} A(h) - \Delta_{\epsilon}^2 h\\
        &= A^2(h)-A(\Delta_{\epsilon})h-\Delta_{\epsilon}A(h)+\Delta_{\epsilon} A(h) - \Delta_{\epsilon}^2 h\\
        &= A^2(h) + \mathcal{K}_{\epsilon}h.
    \end{align*}
\end{proof}

Recall that the notion of Jacobi field is closely related to curvature, on the one hand, and geodesics, on the other. And since $A$ is a geodesic vector field, it is natural to inquire about the Jacobi fields relative to $A$ in the sense of Definition \ref{def:Jacobi}, i.e., vector fields $J$ satisfying equation \eqref{eq:Jacobirel}:
\begin{equation}\label{JacobiDef}
    \nabla_{A} \nabla_{A}J+\mathcal{K}_{\epsilon}\left(J-g(J,A)A\right)=0.
\end{equation}

Observe that it can be checked that the vector fields of the form $J=(ax+b)A$, for $a,b\in \mathbb R$, trivially satisfy equation \eqref{JacobiDef}. We will refer to them as trivial Jacobi fields relative to $A$. The following result establishes the role of non-trivial Jacobi fields in the integrability of \eqref{ODE}.

\begin{theorem}\label{th:Jacobi}
The knowledge of a non-trivial Jacobi vector field relative to $A$ on any of the surfaces $\mathcal{S}_{\epsilon}$ leads to the integration of the first-order ODE \eqref{ODE}.
\end{theorem}

\begin{proof}
Let $J$ be a non-trivial Jacobi field relative to $A$ on $\mathcal{S}_{\epsilon}$, for certain $\epsilon\in \mathcal{C}^{\infty}(U)$. We can decompose $J$ into its components with respect to the frame $\{A,e^{-\epsilon}\partial_u\}$ as follows:
$$
J=\sigma A+\delta e^{-\epsilon}\partial_u,
$$
where $\sigma=g_{\epsilon}(J,A)$ and $\delta=g_{\epsilon}(J, e^{-\epsilon}\partial_u)$. 

Consider, first, the case where $\delta\neq 0$. We assume, without loss of generality, that $\delta$ is non-vanishing in $U$ (if necessary, $U$ can be appropriately reduced). Since $J$ is a Jacobi field relative to $A$, equation \eqref{JacobiDef} holds:
\begin{equation}
    \nabla_{A} \nabla_{A}\left( \sigma A+\delta e^{-\epsilon}\partial_u\right)+\mathcal{K}_{\epsilon} \delta e^{-\epsilon}\partial_u=0.
\end{equation}
Now, since $\nabla_{A} A=\nabla_{A}\left(  e^{-\epsilon}\partial_u\right)=0$, as can be deduced from the connection 1-form \eqref{LCdef}, we obtain 
\begin{equation}
    A^2(\sigma) A+A^2(\delta) e^{-\epsilon}\partial_u+ \mathcal{K}_{\epsilon} \delta e^{-\epsilon}\partial_u=0.
\end{equation}
Therefore, $\delta$ must satisfy
\begin{equation}\label{eq:JacDelta}
A^2(\delta)+\mathcal{K}_{\epsilon} \delta=0.
\end{equation}

According to Lemma \ref{lem:op}, we have that
$$
\left(\mathfrak{T}_{\epsilon}\circ \mathfrak{S}_{\epsilon} \right)(\delta)=0,
$$
and we can use Lemma \ref{prop:intsymop} to conclude that
\begin{itemize}
    \item either $\mathfrak{S}_{\epsilon}(\delta)=0$ and then $e^{\epsilon} \delta^{-1}$ is an integrating factor of \eqref{ODE};
    \item or $\mathfrak{S}_{\epsilon}(\delta)\neq 0$ and $\mathfrak{T}_{\epsilon}(\mathfrak{S}_{\epsilon}(\delta))=0$, so $e^{\epsilon} \mathfrak{S}_{\epsilon}(\delta)$ is an integrating factor for \eqref{ODE}.
\end{itemize}
Therefore, equation \eqref{ODE} can integrate it by quadratures.

Suppose now that $\delta=0$, and therefore $\sigma\neq 0$. If we write equation \eqref{JacobiDef} in this particular case we obtain:
\begin{equation}
    \nabla_{A} \nabla_{A}\left( \sigma A\right)=A^2(\sigma)A=0,
\end{equation}
and $A^2(\sigma)=0$. Consequently, if $A(\sigma)$ is not constant, it is a non-trivial first integral of $A$. 

On the contrary, if $A(\sigma)= a\in \mathbb R$, then $A(\sigma-ax)=0$, and $\sigma-ax$ would be a first integral of $A$. Observe that it is, indeed, a non-trivial first integral, because if 
$$
\sigma-ax=b, b\in\mathbb R,
$$
then $J=(ax+b)A$, and it would be a trivial Jacobi field relative to $A$.

In any case, equation \eqref{ODE} is solved, and the result is proven.
\end{proof}

In conclusion, deforming $\mathcal{S}_0$ into any surface where a non-trivial Jacobi field with respect to $A$ can be found enables the integration of equation \eqref{ODE}. As a consequence, we have the following result:

\begin{theorem}\label{th:curv}
Given the first-order ODE \eqref{ODE}, the deformation of the associated surface $\mathcal{S}_0$ into a surface of constant curvature leads to its integration by quadratures.
\end{theorem}

\begin{proof}
Consider a first-order ODE \eqref{ODE}, and suppose that $\epsilon$ is such that $\mathcal{S}_{\epsilon}$ has constant curvature $\mathcal{K}_{\epsilon}=k\in\mathbb{R}$. According to the calculations in the proof of Theorem \ref{th:Jacobi}, the vector field defined by $J:=\delta e^{-\epsilon}\partial_u$ IS a non-trivial Jacobi field relative to $A$ if and only if
$$
A^2(\delta)+k \delta=0.
$$

We can assume $\delta=\delta(x)$ and solve the second-order ODE
\begin{equation}\label{eqDeltaKcte}
    \delta''+k \delta=0
\end{equation}
to determine $\delta$. It is well known that the general solution to this equation is
\begin{equation}\label{eq:deltaKcte}
    \delta(x)=\begin{cases}
        A \cos\left(\sqrt{k}x\right)+B \sin\left(\sqrt{k}x\right), & \text{ if } k>0,\\
        A+B x, & \text{ if } k=0,\\
        A \cosh\left(\sqrt{-k}x\right)+B \sinh\left(\sqrt{-k}x\right), & \text{ if } k<0,
        \end{cases}
\end{equation}
with $A,B\in\mathbb{R}$.

Then, for any constant curvature, we have a non-trivial Jacobi field $J=\delta e^{-\epsilon}\partial_u$ relative to $A$, and according to Theorem \ref{th:Jacobi} we can integrate equation \eqref{ODE}.
\end{proof}

In the following example we illustrate this result with a first-order ODE whose associated surface can be deformed into a surface of constant positive curvature.

\begin{example}\label{ex:defpos}
Consider the first-order ODE
\begin{equation}\label{odeEx3}
    \dfrac{du}{dx}=u^2,
\end{equation}
whose associated surface $\mathcal{S}_0$ has Gaussian curvature
$$
\mathcal{K}_0=-\partial_u(A(u^2))=-6u^2.
$$

To deform this surface into a constant positive curvature surface we take
$$
\epsilon=\ln\left(\frac{1}{u^2} \sin\left(\frac{1}{u}\right)\right).
$$
In this case, we have $\Delta_{\epsilon}=-\cot\left(\frac{1}{u}\right)$, and then the surface $\mathcal{S}_{\epsilon}$ has curvature
$$
\mathcal{K}_{\epsilon}=-A(\Delta_{\epsilon})-\Delta_{\epsilon}^2=1,
$$
where $A=\partial_x+u^2 \partial_u$.

Therefore, we take the particular solution of \eqref{eqDeltaKcte} given by $\delta(x)=\sin\left(x\right)$, which thus satisfies
$$
A^2(\delta)+\delta=0.
$$

By Lemma \ref{lem:op}, we have that
$$
\left(\mathfrak{T}_{\epsilon}\circ \mathfrak{S}_{\epsilon} \right)(\delta)=\left(\mathfrak{T}_{\epsilon}\circ \mathfrak{S}_{\epsilon} \right)(\sin(x))=0.
$$
 
Finally, since
$$
\mathfrak{S}_{\epsilon} (\sin(x))=A(\sin(x))-\Delta_{\epsilon} \sin(x)=\cos(x)+\cot\left(\frac{1}{u}\right)\sin(x)\neq 0,
$$
we conclude that, necessarily
$$
\mathfrak{T}_{\epsilon}(\mathfrak{S}_{\epsilon} (\sin(x)))=0,
$$
and then, by part \ref{it:integrating} in Lemma \ref{prop:intsymop}, we have that
$$
\begin{aligned}
e^{\epsilon} \mathfrak{S}_{\epsilon} (\sin(x))&=\frac{1}{u^2} \sin\left(\frac{1}{u}\right)\left(\cos(x)+\cot\left(\frac{1}{u}\right)\sin(x)\right)\\
                                        &=\frac{1}{u^2} \sin\left(x+\frac{1}{u}\right)
\end{aligned}
$$
is an integrating factor for \eqref{odeEx3}.
\end{example}

Now we show an example of a first-order ODE whose associated surface can be deformed into a surface of constant negative curvature.
\begin{example}
Consider the first-order ODE
\begin{equation}\label{odeEx4}
    \dfrac{du}{dx}=\dfrac{1-3x u}{x^2}.
\end{equation}
The associated vector field is $A=\partial_x+\frac{1-3x u}{x^2} \partial_u$, and the Gaussian curvature of the associated surface $\mathcal{S}_0$ is
$$
\mathcal{K}_0=-\partial_u\left(A\left(\dfrac{1-3x u}{x^2}\right)\right)=-\dfrac{12}{x^2}.
$$

We can deform it by using $\epsilon=x+3\ln(x)$, in such a way that now $\Delta_{\epsilon}=1$. Then
$$
\mathcal{K}_{\epsilon}=-A(\Delta_{\epsilon})-\Delta_{\epsilon}^2=-1.
$$

In this case a suitable $\delta$ giving rise to a Jacobi field relative to $A$ is 
$$
\delta(x)=\sinh(x),
$$
so by Lemma \ref{lem:op} we have that
$$
\left(\mathfrak{T}_{\epsilon}\circ \mathfrak{S}_{\epsilon} \right)(\sinh(x))=0.
$$

Observe that $\mathfrak{S}_{\epsilon} (\sinh(x))=A(\sinh(x))-\Delta_{\epsilon} \sinh(x)=\cosh(x)-\sinh(x)=e^{-x}$, which is not constant. Since
$$
\mathfrak{T}_{\epsilon}(e^{-x})=0,
$$
we have by Lemma \ref{prop:intsymop} part \ref{it:integrating} that
$$
\begin{aligned}
e^{\epsilon} e^{-x}=e^{x+3\ln(x)}e^{-x}=x^3
\end{aligned}
$$
is an integrating factor for \eqref{odeEx4}.

\end{example}

An immediate conclusion of Theorem \ref{th:curv} is the following corollary, which extends the result of Theorem \ref{th:integr} to the case of surfaces of constant curvature:

\begin{corollary}\label{cor:constant}
If the associated surface $\mathcal{S}_0$ has constant curvature, then the ODE can be integrated by quadratures.
\end{corollary}

To conclude, we will show that the deformation of the associated surface into a surface of zero curvature is the geometric counterpart of the integration of the ODE through the identification of an integrating factor.
\begin{proposition}
    The knowledge of an integrating factor for the first-order ODE \eqref{ODE} is equivalent to the deformation of the associated surface $\mathcal{S}_0$ into a surface of zero curvature.
\end{proposition}
\begin{proof}
Let $\mu$ be an integrating factor for \eqref{ODE}. Then, it satisfies
\begin{equation}\label{eq:mu}
    d(\mu \phi dx-\mu du)=(-\mu_u \phi-\mu \phi_u-\mu_x )dx \wedge du=0.   
\end{equation}
Therefore, we can take $\epsilon=\ln{\mu}$ (where we possibly have to restrict the open set $U$ of $\mu$ to ensure that $\epsilon$ is well defined), and then
$$
\Delta_{\epsilon}=-\mu_u \phi-\mu \phi_u-\mu_x=0,
$$
because of \eqref{eq:mu}. Thus, for this choice of $\epsilon$ we have
$$
\mathcal{K}_{\epsilon}=-A(\Delta_{\epsilon})-\Delta_{\epsilon}^2=0,
$$
so the deformed surface $\mathcal{S}_{\epsilon}$ has zero curvature.
\end{proof}

So, remarkably, we can think of Theorem \ref{th:curv} as a generalization of the integrating factor approach to solving first-order ODEs.

\section{Conclusions}
In this paper, we have studied a particular approach for associating a surface with a given first-order ODE and explored how the integrability of the equation relates to geometric concepts such as geodesics, curvature, and Jacobi fields. 

We have introduced a notion of surface deformation, and showed that the deformation into a constant curvature surface implies integrability. The particular case of the deformation to a zero curvature surface is the geometric counterpart of the knowledge of an integrating factor for the equation.

Furthermore, Theorem~\ref{th:Jacobi} establishes that the knowledge of a non-trivial Jacobi field relative to $A$ on any deformed surface (including the original one) leads to the integration of the equation. Since Lie symmetries also provide a well-known mechanism for integrating an ODE, it is interesting to explore their relationship with Jacobi fields. 

We find it remarkable that a classical geometric notion such as the curvature of the associated surface, or its deformation, is directly related to the integrability of the equation; and we believe that this approach can provide a new perspective on the study of first-order ODEs.
\bibliographystyle{unsrt}  
\bibliography{references.bib}
\end{document}